\documentclass[11pt]{article}
\usepackage{amsmath,amsfonts,amsthm,amscd,amssymb,graphicx}
\usepackage{subfigure}

\numberwithin{equation}{section}

\usepackage{hyperref}
\usepackage{verbatim} 

\newtheorem{theorem}{Theorem}[section]

\newtheorem{lemma}[theorem]{Lemma}

\newtheorem{proposition}[theorem]{Proposition}

\newtheorem{corollary}[theorem]{Corollary}

\newtheorem{remark}[theorem]{Remark}

\newcommand{\RR}{\mathbb{R}}

\def\beq{\begin{equation}}
\def\eeq{\end{equation}}
\def\bb1{{1\!\!1}}
\def\w{{\omega}}

\def\Rt{\mathbb{R}^2}

\def\triangle{\Delta}
\def\bega{\begin{aligned}}
\def\enda{\end{aligned}}
\def\w{\omega}
\def\pt{\partial}
\def\lw{\left}
\def\rw{\right}

\def\wb{{\bar w_2}}
\def\vb{{\bar v_2}}
\def\we{{\bar w_1}}
\def\ve{{\bar v_1}}
\def\wlm{\rightharpoonup}
\def\wtd{\widetilde}

\def\lbb{L^4\cap L^{4/3}}

\begin{document}

\title{The inviscid limit of Navier-Stokes equations for vortex-wave data on $\RR^2$} 

\author{Toan T. Nguyen\footnotemark[1] \and Trinh T. Nguyen\footnotemark[1]
}

\maketitle

\begin{center}
\emph{This paper is dedicated to Walter Strauss
\\on the occasion of his 80th birthday, as token of friendship and admiration.}
\end{center}

\renewcommand{\thefootnote}{\fnsymbol{footnote}}

\footnotetext[1]{Department of Mathematics, Penn State University, State College, PA 16803. Emails: nguyen@math.psu.edu; txn5114@psu.edu.}

\begin{abstract}

We establish the inviscid limit of the incompressible Navier-Stokes equations on the whole plane $\RR^2$ for initial data having vorticity as a superposition of point vortices and a regular component. In particular, this rigorously justifies the vortex-wave system from the physical Navier-Stokes flows in the vanishing viscosity limit, a model that was introduced by Marchioro and Pulvirenti in the early 90s to describe the dynamics of point vortices in a regular ambient vorticity background. 
The proof rests on the previous analysis of Gallay in his derivation of the vortex-point system.

\end{abstract}



\section{Introduction}
In this paper, we are interested in the vanishing viscosity limit of the incompressible Navier-Stokes equations on the plane $\RR^2$ for irregular initial data; namely, we consider 
\beq\label{NS-u}
\begin{aligned}
\partial_t u^\nu + u^\nu \cdot \nabla u^\nu + \nabla p^\nu &= \nu \Delta u^\nu, 
\\
\nabla \cdot u^\nu & =0,
\end{aligned}\eeq
for fluid velocity $u^\nu(x,t) \in \RR^2$ and pressure $p^\nu(x,t) \in \RR$ at $x\in \RR^2$ and $t\ge 0$. The interest is to understand the asymptotic behavior of solutions in the inviscid limit $\nu \to 0$. 

It is straightforward to show that in the absence of spatial boundaries, regular solutions of the Navier-Stokes equations converge in strong Sobolev norms to the regular solutions of Euler equations as $\nu \to 0$ (e.g., \cite{Kato, Swann, Mas}). The convergence (in $L^2$ for velocity fields) also holds for non-smooth solutions that include vortex patches \cite{CWu1,CWu2,Che,Mas,Sueur}. The problem is largely open for less regular data \cite{CS, Chen}, or even for regular data in domains with a boundary (e.g., \cite{SammartinoCaflisch2,Mae,2N,GrNsup} and the references therein). 

For initial data whose vorticity consists of a finite sum of point vortices (Dirac masses), Gallay \cite{Gallay11} proved that the corresponding Navier-Stokes vorticity indeed converges weakly in the inviscid limit to the sum of point vortices whose centers evolve according to the Helmholtz-Kirchhoff point-vortex system. In this paper, we study the case when initial vorticity consists of one point vortex and a regular part. The case of finitely many point vortices can be treated similarly in combination of \cite{Gallay11} where the vortex-point interaction is understood. 

Let us now detail the problem. For velocity field $u^\nu = (u_1^\nu,u_2^\nu)$, let $\w^\nu = \partial_{x_2} u_1^\nu - \partial_{x_1} u_2^\nu$ be the corresponding vorticity. Taking advantage of the divergence-free condition, we can recover the velocity from vorticity through the so-called Biot-Savart law 
\begin{equation}\label{BS-law}
u^\nu=\nabla^\perp \Delta^{-1} \w^\nu = K\star \w^\nu ,\qquad K(x)=\dfrac{1}{2\pi}\dfrac{x^\perp}{|x|^2} ,
\end{equation}
where $K(x)$ denotes the Green kernel of $\nabla^\perp \Delta^{-1}$, the $\star$ notation stands for the usual convolution in variable $x\in \RR^2$, and $a^\perp = (a_2,-a_1)$ for vectors $a \in \RR^2$. It follows from \eqref{NS-u} that the vorticity solves 
\beq \label{NS}
\pt_t \w^\nu+u^\nu \cdot \nabla \w^\nu=\nu \triangle \w^\nu.
\eeq
We solve the vorticity equation \eqref{NS}, together with \eqref{BS-law}, for initial data of the form  
\beq\label{initial}
\w^\nu_{\vert_{t=0}} = \delta_{z_0}(x)+\w_0^E(x),
\eeq  
where $\delta_{z_0}$ denotes the Dirac delta function centered at $x = z_0$ and $\w_0^E$ is the regular component of vorticity that has compact support and vanishes in a neighborhood of $z_0$. The existence and uniqueness for 2D Navier-Stokes equations with such initial data, or in fact more generally, with initial data of finite measures are known; see, for instance, \cite{Cottet, Giga, Kato, GG}. 

\subsection{Vortex-wave system}
In the inviscid limit, we do not expect the limiting solutions from \eqref{NS}-\eqref{initial} to satisfy Euler equations, even in a weak sense\footnote{In fact, it is not known whether weak solutions to Euler equations exist with point vortex data \cite{MaPu91,MaPu94}.}, but rather the following so-called vortex-wave system coined by Marchioro and Pulvirenti \cite{MaPu91,MaPu94} in the early 90s:  
\beq \label{def-VW}
\bega
&\pt_t\w^E+ ( v^E + H)\cdot \nabla \w^E =0\\
&\dot {z}(t)=v^E(t,z(t)),\\
&\w^E_{\vert_{t=0}}=\w^E_0,\qquad z(0)=z_0,\\
\enda
\eeq
in which $v^E = K \star \w^E$ and $H = K(\cdot - z(t))$. That is, in the limit, the regular component of vorticity is transported by the full velocity, while the location of point vortex is propagated by the velocity $v^E$ generated by the regular vorticity $\w^E$. 

The global weak solutions of \eqref{def-VW} in $L^1\cap L^\infty$ were already obtained in \cite{MaPu91,MaPu94} (see also \cite{Miot11,Miot16} for an extension to $L^p$ spaces), while their uniqueness is proved for Lipschitz or even bounded data \cite{Star,Lacave}, provided the ambient velocity is constant in a neighborhood of the point vortex. In particular, let us recall the following theorem. 

\begin{theorem}[\cite{Lacave}] \label{Lacave} Consider initial data $z_0\in \RR$ and $\w_0^E\in L^1\cap L^\infty(\Rt)$. Assume that $\w_0^E$ has compact support and is constant in a neighborhood of $z_0$. Then, there are a unique global solution $(z(t),\w^E(t))$ to \eqref{def-VW} and a positive function $R(t)$ so that $\w^E(t)$ remains constant in the ball centered at the point vortex $z(t)$ with radius $R(t)$ for all times $t\ge 0$. If we assume in addition that $\w_0^E \in W^{k,p}$ for $kp > 2$ and $p>1$, then for any $T\ge 0$, there holds
\begin{equation}\label{VW-smooth} \sup_{0\le t\le T}\| \w^E(t)\|_{W^{k,p}}  \le C_T\end{equation}
for some constant $C_T$. 
\end{theorem}

Theorem \ref{Lacave} assures that $H = K(\cdot - z(t))$ remains regular in the support of $\nabla \w^E(t)$. The stated regularity \eqref{VW-smooth} thus follows from that of Euler equations on $\RR^2$ (\cite{Majda}). 

The vortex-wave system \eqref{def-VW} can be rigorously derived from Euler equations by replacing the initial Dirac mass $\delta_{z_0}$ by $\epsilon^{-2}\chi_{\epsilon}$, for $\chi_\epsilon$ being the characteristic function of the ball $\{|x-z_0|\le \epsilon\}$ and taking $\epsilon \to 0$. This was done in \cite{MaPu93} (see also \cite{Bjor,GLS}). It can also be derived from Navier-Stokes equations in the small viscosity limit, provided that $\nu \le \epsilon^\alpha$ for $\alpha>0$, as done similarly for the vortex-point system \cite{Ma90,Ma98,Ma07}. In this paper, we give a direct derivation of \eqref{def-VW} as the inviscid limit of the Navier-Stokes flows \eqref{NS} with data \eqref{initial}.

\subsection{Main result}

Consider the viscous problem \eqref{NS} with initial data \eqref{initial}. Following \cite{GG,Gallay11}, we first decompose the vorticity into the so-called regular part $\w^{E,\nu}$ and irregular part $\w^{B,\nu}$, both of which are advected by the full velocity vector field $u^\nu = K \star \w^\nu$. Precisely, we write 
\begin{equation}\label{de-w}
\w^\nu=\w^{E,\nu}+\w^{B,\nu},
\end{equation}
where $\w^{E,\nu}$ and $\w^{B,\nu}$ solve 
\beq\label{req}
\bega 
\pt_t\w^{E,\nu}+u^\nu \cdot \nabla \w^{E,\nu}&=\nu\triangle\w^{E,\nu},\\
\w^{E,\nu}|_{t=0}&=\w_0^E,
\enda 
\eeq
and 
\beq \label{irreq}
\bega
\pt_t\w^{B,\nu}+u^\nu\cdot \nabla \w^{B,\nu}&=\nu\triangle \w^{B,\nu},\\
\w^{B,\nu}(t)\quad   &{\rightharpoonup} \quad \delta_{z_0}\quad \text{as}\quad t\to 0^+ .
\enda 
\eeq
Here and in what follows, the weak convergence for finite measures is understood in the following sense: $\mu_n ~{\rightharpoonup} ~\mu$ if and only if 
$$\int_{\RR^2} \phi d\mu_n \to \int_{\RR^2}\phi d\mu,$$
for all the continuous functions $\phi$ that vanish at infinity. 
A direct computation shows that the decomposition preserves the mass: 
\beq\label{mass} 
\int_{\Rt}\w^{E,\nu}(x,t)dx=\int_{\Rt}\w_0^E(x)dx, \qquad \int_{\Rt}\w^{B,\nu}(x,t)dx=1 ,
\eeq
for all positive times. We shall prove that in the inviscid limit $\w^{E,\nu}\to \w^E$ and $\w^{B,\nu}$ is concentrated near the point vortex $z(t)$, transported by $v^E$, yielding weak solutions to the vortex wave system with the same initial data $(\w_0^E,z_0)$. Precisely, our main theorem reads as follows.  

\begin{theorem}\label{mainthm}
Let $z_0\in \RR$ and $\w_0^E\in W^{4,4}(\Rt)$ that has compact support and vanishes in a neighborhood of $z_0$, and let $(z(t),\w^E(t))$ and $\w^\nu(t)$ be the unique solution to the vortex-wave system \eqref{def-VW} and to the Navier-Stokes equation \eqref{NS}, respectively, with initial data $\w_0=\w_0^E+\delta_{z_0}$. Then, there exists a time $T>0$, independent of $\nu$, such that the vorticity $\w^\nu(t)$ can be written as 
\[
\w^\nu(x,t)=\w^{E,\nu}(x,t)+\w^{B,\nu}(x,t),
\]
where $\w^{E,\nu}(t), \w^{B,\nu}(t)$ satisfy
\[
\bega 
\sup_{0\le t\le T}\|\w^{E,\nu}(t)-\w^E(t)\|_{\lbb(\Rt)}&\le C_T \nu, \\
\sup_{0\le t\le T} t^{-1}\lw\|\w^{B,\nu}(t,x)-\frac{1}{4\pi\nu t}e^{-\frac{|x-z(t)|^2}{4\nu t}}\rw\|_{L^1(\Rt)}&\le C_T \nu ,
\enda
\]
for some constant $C_T$ independent of $\nu$. In particular, $\w^{E,\nu}(t) \to \w^E(t)$ strongly in $L^4 \cap L^{4/3}$ and   
$\w^{B,\nu}(t,\cdot)\wlm \delta_{z(t)}(\cdot)$ weakly in the sense of finite measures in the inviscid limit. 
\end{theorem}

Theorem \ref{mainthm} derives the vortex-wave system \eqref{def-VW} as an inviscid limit of Navier-Stokes flows on the whole plane, complementing the earlier derivation \cite{MaPu93,Bjor,GLS} from Euler equations. In addition, we obtain:
$$ T \ge \min \Big\{ T_*^-, \frac{1}{5\|\nabla v^E \|_{L^\infty}} \Big\}$$
for $T_*$ being the smallest time when the point vortex $z(t)$ meets the support of $\w^E(s)$ for some $s\in [0,t]$, recalling from Theorem \ref{Lacave} that $w(t)$ never meets the support of $\w^E(t)$ for all times. 
See Proposition \ref{prop-tWE} and Remark \ref{rem-time1}. 

Let us now discuss some difficulties in proving the theorem. First of all, the initial data containing a Dirac  mass are too singular to perform a direct proof from the standard $L^2$ energy estimates. One then needs to construct a good approximation of solutions to treat the singular part, and control the remainder. The difficulty arises due to the presence of an vortex-wave interaction term of the form 
\beq\label{inter}
v^{E,\nu}(t,x)\cdot \nabla_x \lw(\frac{1}{4\pi\nu t}e^{-\frac{|x-z(t)|^2}{4\nu t}}\rw)
.\eeq
Formally, this term blows up when $x$ is near the point vortex $z(t)$ and $\nu t\to 0$. To treat this singularity, we follow \cite{Gallay11} to work in the vortex scaling variable, construct approximate solutions, and perform weighted energy estimates to control the remainder. However, the weighted energy estimates with the scaling variable $\xi=\frac{x-z(t)}{\sqrt{\nu t}}$ used in \cite{Gallay11} are not enough to treat the interaction term \eqref{inter}, as it leaves a remainder of order one, but not smaller. To overcome this difficulty, we introduce an \textit{approximate viscous} vortex wave system (Section \ref{sec-appVW}), along with the new point vortex $\wtd z(t)=z(t)+O(\nu t)$  and the scaled variable $\xi=\frac{x-\wtd z(t)}{\sqrt{\nu t}}$ in order to close the estimate.

Lastly, we remark that we assume the initial vorticity to be $\delta_{z_0}+\w_0^E$, where $\w_0^E$ is smooth and compactly supported away from the point vortex $z_0$. The regularity is needed in the construction of the high order approximation of solutions. It would be interesting to further combine our analysis with the viscous approximation near vortex-patch solutions constructed in \cite{Sueur} to treat the case when $\w_0^E \in L^1 \cap L^\infty$.

\subsection{Notations}
We will denote $A\lesssim B$ to mean that $|A|\le C_0 |B|$ for some universal constant $C_0>0$ independent of the viscosity $\nu$. We write $f=O(g)$ to mean that $f\lesssim g$, or simply $O(g)$ to mean that the term can be bounded by $C_0 |g|$ for some constant $C_0>0$ independent of $\nu$. 
We define the norm $\|\cdot\|_{\lbb}$ and $\|\cdot\|_{L^1\cap L^\infty}$ of a function $\w(x)$ in $\Rt$ to be 
\[
\|\w\|_{\lbb}=\|\w\|_{L^4}+\|\w\|_{L^{4/3}},\qquad \|\w\|_{L^1\cap L^\infty}=\|\w\|_{L^1}+\|\w\|_{L^\infty}
\]
We also denote by $\mathfrak{m}(\cdot)$ the Lebesgue measure on $\Rt$.

~\\
{\bf Acknowledgement:} The authors would like to thank Thierry Gallay and Christophe Lacave for their many insightful discussions on the subject. The research was supported by the NSF under grant DMS-1764119 and by an AMS Centennial Fellowship. Part of this work was done while the authors were visiting the Department of Mathematics and the Program in Applied and Computational Mathematics at Princeton University.

\section{Approximate vortex wave system}\label{sec-appVW}

Let $(z(t),\w^E)$ be the global solution to the vortex-wave system \eqref{def-VW} with initial data $\w^E_0\in W^{4,4}$ that has compact support and vanishes in a neighborhood of $z_0$.  We introduce an \textit{approximate viscous} vortex-wave system $(\wtd z(t), \wtd\w^E)$, given by
\beq \label{appvor}
\begin{aligned}
\wtd \w^E(x,t)&=\w^E(x,t)+\nu w_{1,a}(x,t)
\\
\pt_t\wtd z=\wtd v^E(\wtd z(t),t) &= K\star \wtd \w^E (\wtd z(t),t), 
\qquad\wtd z(0)=z_0,
\end{aligned}
\eeq
where the added vorticity component $w_{1,a}$ solves 
\beq \label{w1a}
\pt_t w_{1,a}+\lw(v^E+\frac{1}{\sqrt{\nu t}}v^G\lw(\frac{x-z(t)}{\sqrt{\nu t}}\rw)\rw)\cdot\nabla w_{1,a}+v_{1,a}\cdot \nabla \w^E = \triangle \w^E
\eeq
with zero initial data. Here and in what follows, velocity and vorticity are defined through the Biot-Savart law \eqref{BS-law}. For instance, $v_{1,a} = K \star w_{1,a}$ and $v^G(\xi)=\frac{1}{2\pi}\frac{\xi^\perp}{|\xi|^2}(1-e^{-|\xi|^2/4})$.

We obtain the following simple proposition. 

\begin{proposition}\label{prop-tWE} 
Let $T_*$ be defined by 
\begin{equation}\label{def-Tstar} T_* = \inf_{t\ge 0} \Big \{ t~:~z(t) \in \cup_{0\le s\le t}\text{supp}(\w^E(s))\Big\},\end{equation}
with $T_* = \infty$ if $z(t)$ never meets the support of $\w^E(s)$ for $s\in [0,t]$. 
Then, for any $T <T_*$, the unique smooth solution $w_{1,a}(t)$ of \eqref{w1a} exists on $[0,T]$, has compact support, vanishes in a neighborhood of $z(t)$, and satisfies
\beq\label{est-w1a}
\mathfrak{m}\lw(\text{supp}(w_{1,a}(t))\rw)+\|w_{1,a}(t)\|_{W^{2,4}(\Rt)}+ \|\partial_t w_{1,a}(t)\|_{L^\infty(\Rt)}+\|v_{1,a}(t)\|_{W^{2,\infty}(\Rt)}\le C_T,
\eeq
for $t \in [0,T]$ and for some constant $C_T$ independent of $\nu$. In addition, there holds
\begin{equation}\label{est-z2}
|\wtd z(t)-z(t)|\le C_T \nu t\qquad \text{for any}\quad t\in [0,T].
\end{equation}
Here, $\mathfrak{m}$ denotes the Lebesgue measure on $\Rt$.
\end{proposition}

\begin{corollary}\label{cor-tWE} Let $T_*$ be defined as in \eqref{def-Tstar}. 
For any $T<T_*$, $\wtd \w^E(t)$ has compact support, vanishes in a neighborhood of $\wtd z(t)$, and satisfies
\beq\label{est-w1a}
\mathfrak{m}\lw(\text{supp}(\wtd \w^E(t))\rw)+\|\wtd \w^E(t)\|_{W^{2,4}(\Rt)}+ \|\partial_t \wtd \w^E(t)\|_{L^\infty(\Rt)}+\|\wtd v^E(t)\|_{W^{2,\infty}(\Rt)}\le C_T,
\eeq
for $t \in [0,T]$ and for some constant $C_T$ independent of $\nu$. 
\end{corollary}
\begin{proof}
The corollary is a direct consequence of Proposition \ref{prop-tWE} and Theorem \ref{Lacave}. 
\end{proof}

\begin{proof}[Proof of Proposition \ref{prop-tWE}] Recall from Theorem \ref{Lacave} that $\w^E(t)$ has compact support and vanishes in a neighborhood of $z(t)$. This remains valid for $w_{1,a}(t)$ for small times, due to the transport structure of \eqref{w1a}. Precisely, $w_{1,a}(t)$ is supported in $\cup_{0\le s\le t}\text{supp}(\w^E(s))$. Since $z(t) \not \in \text{supp}(\w^E(t))$ for all positive times, we have $T_*>0$ by continuity. Thus, for any $T<T_*$, there is a positive distance $d_T$ so that 
\begin{equation}\label{def-dT}
|x-z(t)|\ge d_T>0 \end{equation}
 for all $x\in \text{supp}\lw(w_{1,a}(t)\rw)$ and $0\le t\le T$, which yields
\[
\lw|\frac{1}{\sqrt{\nu t}}v^G\lw(\frac{x-z(t)}{\sqrt{\nu t}}\rw)\rw|=\frac{1}{2\pi|x-z(t)|}\lw(1-e^{-\frac{|x-z(t)|^2}{4\nu t}}\rw)\le \frac{1}{2\pi|x-z(t)|}\le \frac{1}{2\pi d_T} .
\]
Similar estimates hold for derivatives of $v^G(\cdot)$ for $x$ away from $z(t)$. It follows from \eqref{w1a} that 
\[\bega
\|w_{1,a}(t)\|_{L^4}&\le \int_0^t \lw(\|\triangle \w^E(s)\|_{L^4}+\|v_{1,a}(s)\|_{L^\infty}\|\nabla \w^E(s)\|_{L^4}\rw)ds\\
&\lesssim \int_0^t (1+\|v_{1,a}(s)\|_{L^\infty})ds,
\enda
\]
which yields the estimate on $w_{1,a}$, upon using the elliptic estimate $\|v_{1,a}\|_{L^\infty}\lesssim \|w_{1,a}\|_{\lbb}$ and the fact that $w_{1,a}$ is compactly supported. The derivative estimates follow similarly. 

Finally, let us prove the estimate on $\wtd z(t)$. By definition, we write 
\beq\label{est82}
\begin{cases}
\wtd z(t)&=z_0+\int_0^t \lw(v^E(\wtd z(s),s)+\nu v_{1,a}(\wtd z(s),s)\rw)ds,\\
z(t)&=z_0+\int_0^t v^E( z(s),s)ds,
\end{cases}
\eeq
which gives \beq \label{est61}
\bega
|\wtd z(t)-z(t)|&\le \int_0^t \lw|(v^E(\wtd z(s),s)-v^E(z(s),s))\rw|ds+\nu\int_0^t |v_{1,a}(\wtd z(s),s)|ds\\
&\le \int_0^t \|\nabla v^E(s)\|_{L^\infty} |\wtd z(s)-z(s)| \; ds+\nu t \sup_{0\le s\le t}\|v_{1,a}(s)\|_{L^\infty} .
\enda
\eeq
Applying the Gronwall's lemma gives \eqref{est-z2}. 
\end{proof}

\section{Inviscid limit for the irregular part}\label{sec3}
In this section, we give estimates on the irregular part of vorticity $\w^{B,\nu}$, solving \eqref{irreq}. Let us recall the equation:
\beq\label{req1}
\bega 
\pt_t\w^{B,\nu}+u^\nu \cdot \nabla \w^{B,\nu}&=\nu\triangle\w^{B,\nu},\\
\w^{B,\nu}|_{t=0}&=\delta_{z_0}.
\enda 
\eeq
Here $u^\nu=v^{E,\nu}+v^{B,\nu}$ is the velocity field for the full Navier-Stokes equations. 
Following \cite{Gallay11}, we introduce the change of variables 
$$\xi=\dfrac{x-\wtd z(t)}{\sqrt{\nu t}}$$
and write 
\beq\label{change}
\begin{aligned}
v^{B,\nu}(x,t)&=\frac{1}{\sqrt{\nu t}}v_2(\xi,t),
\qquad \w^{B,\nu}(x,t)=\frac{1}{\nu t}w_2(\xi,t).\\
\end{aligned}
\eeq
Here, we recall that $\wtd z(t)$ to be the solution to the approximate vortex wave system, given in \eqref{appvor}.
Note that the change of variables is consistent with the Biot-Savart law: $v_2 = K \star_\xi w_2$. Putting the Ansatz into the equation \eqref{irreq} for $\w^{B,\nu}$, we get the following equation 
\beq \label{ir-eq}
\bega
&\Phi(w_2,v^{E,\nu}): = \lw(t\pt_t -\mathcal{L}\rw)w_2 +\sqrt{\dfrac{t}{\nu}} (v^{E,\nu}(\wtd z(t)+ \xi \sqrt{\nu t},t)- \partial_t \wtd z(t))\cdot \nabla_\xi w_2+\frac{1}{\nu}v_2\cdot \nabla_\xi w_2 =0,
\enda
\eeq
where $\mathcal{L}$ is defined by 
$$\mathcal{L}w_2:=\triangle_\xi w_2+\frac{1}{2}\xi\cdot \nabla_\xi w_2+w_2.$$

In the vanishing viscosity limit, we expect that the viscous regular velocity remains close to the inviscid one: $v^{E,\nu} \to v^E$, and hence the irregular part should tend to the so-called Lamb-Oseen vortex, which is defined by 
 \[
 G(\xi)=\dfrac{1}{4\pi}e^{-|\xi|^2/4},\qquad v^G(\xi)=\dfrac{1}{2\pi}\dfrac{\xi^\perp}{|\xi|^2} \Big(1-e^{-|\xi|^2/4}\Big)
 . \]
It follows that $\mathcal{L}G = 0$ and $v^G \cdot \nabla_\xi G = 0$. Therefore, the pair $(G(\xi),v^{E,\nu})$ solves \eqref{ir-eq}, up to the following error term $t R_1(\xi,t)$, with 
\beq\label{def-R1}
R(\xi,t): =\frac{1}{\sqrt{\nu t}}\lw( v^{E,\nu}(\wtd z(t) + \xi\sqrt{\nu t},t) - \wtd v^E(\wtd z(t),t)\rw)\cdot \nabla G 
\eeq
which does not vanish in the inviscid limit, upon recalling that $\pt_t\wtd z(t)=\wtd v^E(\wtd z(t),t)$. 
Roughly speaking, $R= \mathcal{O}(1)$ in the small viscosity limit.  

We shall construct better approximate solutions to the equation \eqref{ir-eq}. Here we stress that the equation \eqref{ir-eq} involves two unknown functions $w_2, v^{E,\nu}$ which are coupled through the full velocity $u^\nu$. To leading order, let us take $v^{E,\nu}_{app}=\wtd v^E$ for $\wtd v^E$ solving the approximate vortex-wave system \eqref{appvor} and 
\beq \label{guess0}
w_{2,\text{app}}(\xi,t)=G(\xi)+(\nu t) w_{2,a}(\xi,t)
\eeq 
where $w_{2,a}$ to be defined later. The pair $(w_{2,\text{app}},v^{E,\nu}_{app})$ thus solves \eqref{ir-eq}, leaving an error of the form  
\begin{equation}\label{1st-app}
\bega
\Phi(w_{2,\text{app}},v^{E,\nu}_{app})
&=t (\Lambda +\nu (1-\mathcal{L}))w_{2,a} + \nu t^2 \partial_t w_{2,a} +\nu t^2 v_{2,a}\cdot \nabla w_{2,a}
\\&\quad +\sqrt{\nu}t^{3/2}(\wtd v^{E}(\wtd z(t)+ \xi \sqrt{\nu t},t)-\wtd v^E(\wtd z(t),t))\cdot \nabla w_{2,a} + t R_1(\xi,t),
\enda 
\end{equation}
where $R_1(\xi,t)$ is defined as in \eqref{def-R1} with $v^{E,\nu}_{app} = \wtd v^E$, and 
\[
\bega
\Lambda w&:=v^G\cdot \nabla_\xi w+v\cdot \nabla_\xi G, \qquad v = K \star w.
\enda
\]
To treat the order one remainder $R_1(\xi,t)$, we first solve $(\Lambda +\nu (1-\mathcal{L}))w_{2,a}  = - R_1$ to leading order in $\nu$. We recall the following proposition from \cite{Gallay11}, Lemma 5 and Remark 1.  

\begin{proposition}\label{prop-Gallay11} 
Let $z=z(\xi)$ be a function of the form
\[
z(\xi)=a_1(r)\cos(2\theta)+a_2(r)\sin(2\theta)+a_3(r)\cos(3\theta)+a_4(r)\sin(3\theta)
\]
for $\xi=re^{i\theta}$. Assume that the coefficients satisfy \[
\sum_{i=1}^4(|a_i(r)|+|a_i'(r)|)\le C_0 P(r)e^{-r^2/4}\qquad \forall r>0.
\]
for some polynomial $P(r)$.
Then for any $\nu>0$, there exists a unique solution $w^\nu$ to the elliptic equation 
\[
\Lambda w^{\nu}+\nu(1-\mathcal L)w^\nu=z
\]
such that 
\[
|w^\nu(\xi)|+|\nabla w^\nu(\xi)|\le C_\gamma e^{-\gamma |\xi|^2/4}
\]
for any $\gamma\in (0,1)$ and for some constant $C_\gamma$ that is independent of $\nu$.
\end{proposition}

\subsection{Vortex-wave reaction term}

In this section, we show that the leading term in the reaction term in \eqref{def-R1} satisfies the assumption of Proposition \ref{prop-Gallay11}. Precisely, we introduce 
\beq \label{Aterm}
R_1(\xi,t)=\frac{1}{\sqrt{\nu t}}(\wtd v^E(\wtd z(t)+\xi\sqrt{\nu t},t)-\wtd v^E(\wtd z(t),t))\cdot \nabla G. 
\eeq
We have the following lemma.

\begin{lemma}\label{A-A0} For any $T>0$, there is a constant $C_T$ so that 
\[
|R_1(\xi,t)-A_0(\xi,t)|\le C_T (\nu t)|\xi|^4e^{-|\xi|^2/4} ,
\]
where 
\beq \label{A0}
\begin{aligned}A_0(\xi,t)&=\frac{1}{16\pi^2}|\xi|^2 e^{-|\xi|^2/4}\int_{\Rt} \frac{\sin(2\psi)}{|\wtd z(t)-y|^2}\wtd \w^E(y,t)dy 
\\&\quad - \frac{1}{16\pi^2} \sqrt{\nu t}|\xi|^3e^{-|\xi|^2/4}\int_{\Rt}\dfrac{\sin(3\psi)}{|\wtd z(t)-y|^3}\wtd \w^E(y,t)dy.
\end{aligned}\eeq
Here, $\psi$ 
denotes the angle between $\xi$ and $\wtd z(t)-y$. 
\end{lemma}

\begin{proof}
Recalling \eqref{Aterm} and $G=\frac{1}{4\pi} e^{-|\xi|^2/4}$, and using the Biot-Savart law \eqref{BS-law}, we have 
\[\bega 
R_1(\xi,t)&=\frac{-1}{8\pi\sqrt{\nu t}}(\wtd v^E(\wtd z(t)+\xi\sqrt{\nu t},t)-\wtd v^E(\wtd z(t),t))\cdot \xi e^{-|\xi|^2/4}\\
&=\frac{-e^{-|\xi|^2/4}}{16\pi^2\sqrt{\nu t}}\int_{\mathbb{R}^2}\xi\cdot \lw(\dfrac{(\wtd z(t)+\xi\sqrt{\nu t}-y)^\perp}{|\wtd z(t)+\xi\sqrt{\nu t}-y|^2}-\dfrac{(\wtd z(t)-y)^\perp}{|\wtd z(t)-y|^2}\rw)\wtd \w^E(y,t)dy\\
&=\frac{-e^{-|\xi|^2/4}}{16\pi^2\sqrt{\nu t}}\int_{\mathbb{R}^2}\xi\cdot (\wtd z(t)-y)^\perp \lw(\dfrac{1}{|\wtd z(t)+\xi\sqrt{\nu t}-y|^2}-\dfrac{1}{|\wtd z(t)-y|^2}\rw) \wtd \w^E(y,t)dy
\\
&=: A_1(\xi,t)+A_2(\xi,t),
\enda 
\]
where $A_1(\xi,t), A_2(\xi,t)$ denote the integral over $\{|\xi|\sqrt{\nu t}\le \frac{1}{2}|\wtd z(t)-y|\}$ and $\{|\xi|\sqrt{\nu t}\ge \frac{1}{2}|\wtd z(t)-y|\}$, respectively. 
Let us first treat $A_1(\xi,t)$. Applying Lemma \ref{exp} for $|\xi|\sqrt{\nu t} \le \frac12|\wtd z(t)-y|$,  we have 
\[
\dfrac{1}{|\wtd z(t)+\xi\sqrt{\nu t}-y|^2}-\dfrac{1}{|\wtd z(t)-y|^2}=\dfrac{1}{|\wtd z(t)-y|^2}\sum_{n=1}^\infty (-1)^n \dfrac{|\xi|^n \sqrt{\nu t}^n}{|\wtd z(t)-y|^n}\dfrac{\sin((n+1)\psi)}{\sin(\psi)}.
\]
Here $\psi$ is the angle between $\xi$ and $\wtd z(t)-y$.
Thus we get 
$$\bega 
&\quad \xi\cdot(\wtd z(t)-y)^\perp\lw(\dfrac{1}{|\wtd z(t)+\xi\sqrt{\nu t}-y|^2}-\dfrac{1}{|\wtd z(t)-y|^2}\rw)\\
&= \sum_{n=2}^\infty (-1)^{n+1}(\nu t)^{\frac{n-1}{2}}\dfrac{|\xi|^n}{|\wtd z(t)-y|^n}\sin(n\psi)\\
&=-(\nu t)^{1/2}\dfrac{|\xi|^2}{|\wtd z(t)-y|^2}\sin(2\psi)+(\nu t)\dfrac{|\xi|^3}{|\wtd z(t)-y|^3}\sin (3\psi)
+
\frac{1}{\sqrt{\nu t}}\sum_{n\ge 4}(-1)^{n+1}\frac{(|\xi|\sqrt{\nu t})^n}{|\wtd z(t)-y|^n}\sin (n \psi),
\enda 
$$
in which we can estimate 
$$\Big|\frac{1}{\sqrt{\nu t}}\sum_{n\ge 4}(-1)^{n+1}\frac{(|\xi|\sqrt{\nu t})^n}{|\wtd z(t)-y|^n}\sin (n \psi) \Big|\le 2 \frac{ (\nu t)^{3/2}|\xi|^4}{|\wtd z(t)-y|^4} ,$$
since $|\xi|\sqrt{\nu t} \le \frac12|\wtd z(t)-y|$. Hence, we have 
\[\bega
A_1(\xi,t)&=\frac{|\xi|^2 e^{-|\xi|^2/4}}{16 \pi^2}\int_{|\xi|\sqrt{\nu t}\le \frac{1}{2}|\wtd z(t)-y|}\dfrac{1}{|\wtd z(t)-y|^2}\sin(2\psi)\wtd \w^E(y,t)dy\\
&\quad -\frac{\sqrt{\nu t}|\xi|^3e^{-|\xi|^2/4}}{16 \pi^2}\int_{|\xi|\sqrt{\nu t}\le \frac{1}{2}|\wtd z(t)-y|}\dfrac{1}{|\wtd z(t)-y|^3}\sin (3\psi)\wtd\w^E(y,t)dy\\
&\quad+\mathcal{O}(\nu t |\xi|^4e^{-|\xi|^2/4})\int_{|\xi|\sqrt{\nu t}\le \frac{1}{2}|\wtd z(t)-y|}\frac{1}{|\wtd z(t)-y|^4}\sin(4\psi)\wtd \w^E(y,t)dy.
\enda
\]
We note that all the integrals above are bounded by $\|\wtd \w^E(t)\|_{L^1}$, since $\wtd z(t)$ is bounded away from the support of $\wtd\w^E(t)$ by Corollary \ref{cor-tWE}. Therefore, defining $A_0(\xi,t)$ as in \eqref{A0}, we can write 
\[\bega
A_1(\xi,t)&=A_0(\xi,t) - \frac{|\xi|^2 e^{-|\xi|^2/4}}{16 \pi^2}\int_{|\xi|\sqrt{\nu t}\ge \frac{1}{2}|\wtd z(t)-y|}\dfrac{1}{|\wtd z(t)-y|^2}\sin(2\psi)\wtd \w^E(y,t)dy\\
&\quad +\frac{\sqrt{\nu t}|\xi|^3e^{-|\xi|^2/4}}{16 \pi^2}\int_{|\xi|\sqrt{\nu t}\ge \frac{1}{2}|\wtd z(t)-y|}\dfrac{1}{|\wtd z(t)-y|^3}\sin (3\psi)\wtd\w^E(y,t)dy
+\mathcal{O}(\nu t |\xi|^4e^{-|\xi|^2/4}) .\enda
\]

It remains to treat the integral over the domain $\{|\xi|\sqrt{\nu t}>\frac{1}{2}|\wtd z(t)-y|\}$. Since $\wtd z(t)$ is bounded away from the support of $\wtd\w^E(t)$, the above (explicitly written) integrals vanish for $|\xi|\sqrt{\nu t} \le c_T$ for all $t\in [0,T]$, for some constant $c_T$. On the other hand, for $|\xi|\sqrt{\nu t} \ge c_T$, we have 
$$
\begin{aligned} \Big| &\frac{|\xi|^2 e^{-|\xi|^2/4}}{16 \pi^2}\int_{|\xi|\sqrt{\nu t}\ge \frac{1}{2}|\wtd z(t)-y|}\dfrac{1}{|\wtd z(t)-y|^2}\sin(2\psi)\wtd \w^E(y,t)dy\Big| 
\le C_T \nu t |\xi|^4 e^{-|\xi|^2/4} \| \wtd \w^E(t)\|_{L^1}, 
\end{aligned} $$
for some constant $C_T$. Similarly, we also have $A_2(\xi,t) =0$ for $|\xi|\sqrt{\nu t} \le c_T$ for all $t\in [0,T]$, for some constant $c_T$, while for $|\xi|\sqrt{\nu t} \ge c_T$, we have  
\[\bega
|A_2(\xi,t)|&\le |A_1(\xi,t)|+|A(\xi,t)| 
\\&\le C_T |\xi|^2 (1 + \nu t |\xi|^2) e^{-|\xi|^2/4} \| \wtd \w^E(t)\|_{L^1}+C_T (\nu t)^{-1/2} |\xi| e^{-|\xi|^2/4} \|\wtd v^E\|_{L^\infty} \\
&\le C_T (\nu t) |\xi|^4 e^{-|\xi|^2/4} ,
\enda
\]
upon using Corollary \ref{cor-tWE} to bound $\wtd v^E$ and $\wtd \w^E$. 
The lemma follows. \end{proof}

\subsection{Construction of an approximation solution}

We now construct $w_{2,a}$ that solves the following elliptic equation \beq \label{solapp}\Lambda w_{2,a}+\nu(1-\mathcal{L})w_{2,a}= -A_0(\xi,t)
\eeq
with $A_0(\xi,t)$ defined as in \eqref{A0}. We have the following. 

\begin{lemma}\label{lem-w2a}
There exists a solution $w_{2,a}$ to \eqref{solapp} so that, for any $\gamma\in (0,1)$, there holds 
\[
|w_{2,a}(t,\xi)|+|\nabla w_{2,a}(\xi,t)|\le  C_\gamma e^{-\gamma |\xi|^2/4}
\]
uniformly in $\nu>0$. In particular, we have 
\beq\label{bound-w12a}
\|v_{2,a}(t)\|_{L^\infty}+\int_{\Rt}|w_{2,a}(\xi,t)|^2 e^{|\xi|^2/4}d\xi+\int_{\Rt}|\nabla w_{2,a}(\xi,t)|^2e^{|\xi|^2/4}d\xi \lesssim 1. 
\eeq
\end{lemma}
\begin{proof} For each $y \in \RR^2$, we introduce 
\beq \label{A0re}
\begin{aligned}B_0(\xi,y,t)&=\frac{-1}{16\pi^2}|\xi|^2 e^{-|\xi|^2/4} \frac{\sin(2\psi)}{|\wtd z(t)-y|^2}\wtd \w^E(y,t)
+ \frac{1}{16\pi^2} \sqrt{\nu t}|\xi|^3e^{-|\xi|^2/4} \dfrac{\sin(3\psi)}{|\wtd z(t)-y|^3}\wtd \w^E(y,t),
\end{aligned}\eeq
recalling $\psi$ the angle between $\xi$ and $\wtd z(t) -y$. If follows from \eqref{A0} that $A_0(\xi,t) = \int_{\RR^2} B_0(\xi,y,t)\; dy$. It is clear that for each $y$, $B_0(\xi,y,t)$ satisfies the assumption of Proposition \ref{prop-Gallay11} and hence we can define 
$$W_{2,a}(\xi,y,t) := \Big( \Lambda +\nu(1-\mathcal{L}) \Big)^{-1} B_0(\xi,y,t) ,$$
stressing that $y\in \RR^2$ and $t\ge 0$ play a role as independent parameters. The solution $w_{2,a}$ is thus defined by the average of $W_{2,a}(\xi,y,t)$ with respect to $y$. 
The pointwise estimates follow directly from Proposition \ref{prop-Gallay11} and the estimates on $\wtd \w^E$. Taking $\gamma>1/2$ and using the elliptic estimate $\| v_{2,a}\|_{L^\infty} \lesssim \| w_{2,a}\|_{L^1\cap L^\infty}$, we obtain the estimates \eqref{bound-w12a}. 
\end{proof}

\subsection{Estimating the error term}
Construct $w_{2,a}$ as in Lemma \ref{lem-w2a}. Then, $w_{2,\text{app}} = G(\xi)+\nu tw_{2,a}$ and $v^{E,\nu}_{\text{app}}=\wtd v^E$ approximately solves \eqref{ir-eq} in the following sense.  

\begin{proposition}
 For any $\gamma\in(0,1)$, there holds 
\beq\label{error}
\lw|\Phi(w_{2,\text{app}},v^{E,\nu}_{app})(\xi,t)\rw|\le C_\gamma \nu t^{3/2} e^{-\gamma |\xi|^2/4}
\eeq 
for some constant $C_\gamma$.
\end{proposition}
\begin{proof} 
Fix a $\gamma \in (0,1)$. Using \eqref{solapp} into \eqref{1st-app}, we write 
$$\bega
\Phi(w_{2,\text{app}},v^{E,\nu}_{app})(\xi,t)
&=\nu t^2 v_{2,a}\cdot \nabla w_{2,a}
 +\sqrt{\nu}t^{3/2}(\wtd v^{E}(\wtd z(t)+ \xi \sqrt{\nu t},t)- \wtd v^E(\wtd z(t),t))\cdot \nabla w_{2,a} 
 \\&\quad + \nu t^2 \partial_t w_{2,a}+ t ( R_1(\xi,t) - A_0(\xi,t))
 \\&=: \sum_{i=1}^4\Phi_i(\xi,t) .
\enda 
$$
Let us estimate each term on the right. 
Using Proposition \ref{prop-tWE} and Lemma \ref{lem-w2a}, we get 
\[
|\Phi_1(\xi,t)|\le \nu t^2 \|v_{2,a}(t)\|_{L^\infty}|\nabla w_{2,a}(\xi,t)|  \lesssim \nu t^2 e^{-\gamma |\xi|^2/4} .
\]
Similarly, using Corollary \ref{cor-tWE}, we bound 
\[
|\wtd v^E(\xi\sqrt{\nu t}+\wtd z(t),t)-\wtd v^E(\wtd z(t),t)|\lesssim |\xi|\sqrt{\nu t} \|\nabla \wtd v^E \|_{L^\infty}
\]
and hence 
\[\bega
|\Phi_2(\xi,t)|&\le \sqrt{\nu} t^{3/2}|\wtd v^E(\xi\sqrt{\nu t}+\wtd z(t),t)-\wtd v^E(\wtd z(t),t)||\nabla w_{2,a}(\xi,t)|\\
&\lesssim \nu t^2 |\xi| e^{-\gamma'|\xi|^2/4} \\
&\lesssim \nu t^2e^{-\gamma|\xi|^2/4},
\enda
\]
upon taking $\gamma'$ from Lemma \ref{lem-w2a} so that $\gamma' > \gamma$. 

Next, we treat $\Phi_3(\xi,t)=\nu t^2\pt_tw_{2,a}$. Since $\sqrt{t}\pt_t$ commutes with $\Lambda$ and $\mathcal{L}$, the equation \eqref{solapp} gives 
\[
\lw(\nu(1-\mathcal L)+\Lambda\rw)(\sqrt t\pt_tw_{2,a})=-\sqrt t\pt_tA_0(\xi,t) .
\]
To apply Proposition \ref{prop-Gallay11}, it suffices to prove that 
\beq\label{bd-dtA0}
\sqrt{t}|\pt_t A_0(\xi,t)|\lesssim |\xi|^2 (1+ |\xi|)e^{-|\xi|^2/4} .
\eeq
Indeed, we recall from \eqref{A0re} that 
\beq\label{est11}
\begin{cases}
A_0(\xi,t)&=\int_{\Rt}B_0(\xi,y,t)dy\\
B_0(\xi,y,t)&=\frac{-1}{16\pi^2}|\xi|^2 e^{-|\xi|^2/4} \frac{\sin(2\psi)}{|\wtd z(t)-y|^2}\wtd \w^E(y,t)
+ \frac{1}{16\pi^2} \sqrt{\nu t}|\xi|^3e^{-|\xi|^2/4} \dfrac{\sin(3\psi)}{|\wtd z(t)-y|^3}\wtd \w^E(y,t),
\end{cases}
\eeq 
where $\psi$ is the angle between $\xi$ and $\wtd z(t)-y$. By Corollary \ref{cor-tWE}, $\wtd \w^E(t) $ and $\partial_t \wtd \w^E(t)$ are both bounded, compactly supported, and vanishing in a neighborhood of $\wtd z(t)$. In particular, $|\wtd z(t) - y|$ is bounded below away from zero for $y$ in the support of $\wtd \w^E(t) $. The estimate \eqref{bd-dtA0} thus follows, upon recalling that $\partial_t \wtd z(t) = \wtd v^E(\wtd z(t),t)$ and $\wtd v^E$ is bounded (Corollary \ref{cor-tWE}). Arguing similarly as in Lemma \ref{lem-w2a}, we obtain 
$$ |\sqrt t\pt_tw_{2,a}(\xi,t)|\le C_\gamma e^{-\gamma |\xi|^2/4}.$$
Finally, the last term $\Phi_4(\xi,t)=t(R_1(\xi,t)-A_0(\xi,t))$ is already treated in Lemma \ref{A-A0}. 
This concludes the proof.
\end{proof}

\subsection{Equations for the remainder}

Having introduced the approximate solutions $w_{2,\mathrm{app}} $ and $v^{E,\nu}_{\mathrm{app}}$, let us now study the remainder. Precisely, we search for solutions of \eqref{ir-eq} in the following form 
\beq  \label{remainder}
\begin{cases}
w_2&=G(\xi)+(\nu t) w_{2,a} +(\nu t)\wb\\
v^{E,\nu}&=\wtd v^E+\nu^{3/2}\ve,
\end{cases}
\eeq 
in which $\wtd v^E$ and $w_{2,a}$ are constructed in the previous sections. Putting this Ansatz into \eqref{ir-eq}, 
we have 
\beq \label{ns2}
\bega
&(t\pt_t -\mathcal L + 1 )\wb+\frac{1}{\nu}\Lambda \wb+\sqrt{\frac{t}{\nu}} (\wtd v^E-\dot{\wtd z})\cdot\nabla \wb
 +t(\vb\cdot \nabla w_{2,a}+v_{2,a}\cdot\nabla \wb) 
 \\& +\frac{1}{\sqrt t}(\ve\cdot\nabla  G) +\nu \sqrt t(\ve\cdot \nabla w_{2,a})
+t(\vb\cdot \nabla \wb)+\nu \sqrt t(\ve\cdot \nabla \wb) + \frac{1}{\nu t}\Phi(w_{2,app}, v^{E,\nu}_{app}) = 0,
\enda 
\eeq
in which we stress that $\wtd v^E$ and $\ve$ are functions of $(x,t)$, while $G, w_{2,a},$ and $\wb$ are functions of $\xi,t$. Again, velocity and vorticity are defined through the Biot-Savart law in their respective variables. 

Our goal is to derive estimates for the remainder solution $(\wb,\ve)$ in suitable function spaces. Precisely, we shall work with the following weighted $L^2$ norm 
$$ \| \omega \|_{L^2_p}^2: = \int_{\Rt}|\omega(\xi) |^2 p(\xi)d\xi, \qquad p(\xi) = e^{|\xi|^2/4}.$$
The weight function is natural in view of the following lemma. 

\begin{lemma}\label{lem-Lambda} The operator $\mathcal{L}$ is self-adjoint in $L^2_p$, while $\Lambda$ is skew-symmetric in $L^2_p$. In particular, $\mathcal{L} \le 0$, we have 
$$\langle \Lambda \omega, \omega \rangle_{L^2_p}=0$$
for any $\omega(\xi)$ in the domain of $\Lambda$.  
\end{lemma}
\begin{proof}
The lemma follows from a direct calculation; see \cite[Lemma 4.8]{GallayW}.
\end{proof}

\begin{lemma}[Elliptic estimates] \label{elip}Let $\vb=K\star_\xi \wb$ be the velocity obtained from $\wb$ by the Biot-Savart law. There holds 
\[
\|\vb\|_{L^\infty}\lesssim \|\wb\|_{L^2_p}+\|\wb\|_{L^2_p}^{1/2}\|\nabla \wb\|_{L^2_p}^{1/2} .
\]
\end{lemma}
\begin{proof}
By H\"older inequality and Sobolev embeddings, we have 
\[\bega
\|\vb\|_{L^\infty}&\lesssim \|\wb\|_{L^{4/3}}^{1/2}\|\wb\|_{L^4}^{1/2}\lesssim \|\wb\|_{L^2_p}^{1/2}\lw(\|\wb\|_{L^2_p}+\|\nabla \wb\|_{L^2_p}\rw)^{1/2}\\
&\lesssim \|\wb\|_{L^2_p}^{1/2}\lw(\|\wb\|_{L^2_p}^{1/2}+\|\nabla \wb\|_{L^2_p}^{1/2}\rw)\\
&=\|\wb\|_{L^2_p}+\|\wb\|_{L^2_p}^{1/2}\|\nabla \wb\|_{L^2_p}^{1/2}.\\
\enda
\]
The proof is complete.
\end{proof}

\subsection{Estimates for the remainder}\label{aest}

This section is devoted to prove the following proposition. 
\begin{proposition}\label{Eest}
There are a positive constant $\kappa$ and a positive time $T$ so that 
\beq \label{apriori}
\bega
t\frac{d}{dt}\|\wb(t)\|_{L^2_p}^2  &+\kappa(\|(1+|\xi|)\wb(t)\|_{L^2_p}^2 + \|\nabla \wb(t)\|_{L^2_p}^2)\\
&\lesssim 
 t \|\wb(t)\|_{L^2_p}^5 + \nu t \|\ve(t)\|_{L^\infty}^4+  t^{-1}\|\ve(t)\|^2_{L^\infty}
\enda
\eeq
uniformly in $\nu$ and in $t\in [0,T]$. 
\end{proposition}

The proposition follows from weighted energy estimates. To proceed, using the equation \eqref{ns2} for $t\pt_t \wb$, we compute   
\beq\label{id-EEE}\bega
t\frac{d}{dt}\| \wb(t) \|_{L^2_p}^2 &=\int_{\Rt}(t\pt_t\wb(\xi,t))\wb(\xi,t) p(\xi)d\xi =\sum_{i=1}^9\mathcal{E}_i(t),
\enda
\eeq
where
\[
\begin{cases}
\mathcal{E}_1(t)&=\int_{\Rt}p(\xi)(\mathcal{L}\wb-\wb)(\xi,t)d\xi,\\
\mathcal{E}_2(t)&=- \frac1\nu\int_{\Rt}\Lambda\wb (\xi,t)\wb(\xi,t)p(\xi)d\xi,\\
\mathcal{E}_3(t)&=-\sqrt{\frac{t}{\nu}}\int_{\Rt} ( (\wtd v^E-\dot{\wtd z}) \cdot \nabla \wb )\wb(\xi,t)p(\xi)d\xi,\\
\mathcal{E}_4(t)&=-t\int_{\Rt}(\vb\cdot \nabla w_{2,a}+v_{2,a}\cdot \nabla \wb)\wb(\xi,t)p(\xi)d\xi,\\
\mathcal{E}_5(t)&=-t\int_{\Rt}\lw(\vb\cdot \nabla\wb\rw)\wb(\xi,t)p(\xi)d\xi,\\
\mathcal{E}_6(t)&=-\nu \sqrt t\int_{\Rt}\lw( \ve\cdot \nabla\wb\rw)\wb(\xi,t)p(\xi)d\xi,\\
\mathcal{E}_7(t)&=-\frac{1}{\nu t}\int_{\Rt}\Phi_{\text{app}}(\xi,t)\wb(\xi,t)p(\xi)d\xi,\\
\mathcal{E}_8(t)&=-\frac{1}{\sqrt{t}}\int_{\Rt}(\ve\cdot \nabla G)\wb(\xi,t)p(\xi)d\xi
,\\
\mathcal{E}_9(t)&=-\nu \sqrt t\int_{\Rt}\lw( \ve\cdot \nabla w_{2,a}\rw)\wb(\xi,t)p(\xi)d\xi .\\
\end{cases}
\]

~\\
Let us estimate each term $\mathcal{E}_i$. Thanks to Lemma \ref{lem-Lambda}, we have $\mathcal{E}_2(t) =0$, while $\mathcal{E}_1(t) \le -\| \wb(t)\|^2_{L^2_p}$. In fact, the following lemma gives a better coercive estimate for $\mathcal{E}_1(t)$.   

\begin{lemma}[Diffusive term]
There holds 
\[
\mathcal{E}_1(t)\le - \frac1{24} \Big(\|\nabla \wb(t) \|^2_{L^2_p} + \|(1 + |\xi|) \wb(t) \|^2_{L^2_p}\Big) .
\]
\end{lemma}
\begin{proof}
Recalling $\mathcal{L}=1+\frac{1}{2}\xi\cdot \nabla +\triangle $ and integrating by parts, we compute \[\bega
&\quad \int_{\Rt}(\mathcal{L}\wb-\wb)(\xi,t)p(\xi)\wb(\xi,t)d\xi\\
&=\int_{\Rt}\lw(\triangle \wb+\frac{1}{2}\xi\cdot \nabla \wb\rw)\wb(\xi,t)p(\xi)d\xi, \\
&=-\int_{\Rt}|\nabla \wb|^2p(\xi)d\xi-\int_{\Rt}\wb(\nabla p\cdot \nabla \wb)d\xi+\frac{1}{4}\int_{\Rt}\lw(\xi\cdot \nabla (|\wb|^2)\rw)p(\xi,t)d\xi\\
&=-\int_{\Rt}|\nabla \wb|^2p(\xi,t)d\xi-\int_{\Rt}\wb(\nabla p\cdot \nabla \wb)d\xi-\frac{1}{2}\int_{\Rt}|\wb|^2p(\xi,t)d\xi-\frac{1}{4}\int_{\Rt}|\wb|^2(\xi\cdot \nabla p)d\xi.\\
\enda
\]
The second integral is treated by 
\[
-\int_{\Rt}\wb(\nabla p\cdot \nabla \wb)d\xi\le \dfrac{3}{4}\int_{\Rt}|\nabla \wb|^2p(\xi,t)+\frac{1}{3}\int_{\Rt}\dfrac{|\nabla p|^2}{p^2}|\wb|^2 p(\xi)d\xi .
\]
Recalling now the weight function $p(\xi) = e^{|\xi|^2/4}$, we obtain the lemma at once. 
\end{proof}

 \begin{lemma}\label{lem-E3} There holds
\[
\mathcal{E}_3(t)\lesssim t \| \xi \wb (t)\|_{L^2_p}^2. 
\]
\end{lemma}
\begin{proof}
Integrating by parts and using the fact that $\wtd v^E-\dot{\wtd z}$ is divergence free, we have 
\[\bega
\mathcal{E}_3(t)&=-\sqrt{\dfrac{t}{\nu}}\int_{\Rt}\lw((\wtd v^E-\dot{\wtd z})\cdot \nabla \bar w_2\rw)\bar w_2(\xi,t)p(\xi)d\xi\\
&=\frac{1}{2}\sqrt{\frac{t}{\nu}}\int_{\Rt}(\wtd v^E-\dot{\wtd z})\cdot \nabla p(\xi)|\bar w_2(\xi,t)|^2d\xi.\\
\enda 
\]
Recalling $\dot{\wtd z} = \wtd v^E(\wtd z(t),t)$ and using Corollary \ref{cor-tWE}, we estimate 
\[
|\wtd v^E(\xi\sqrt{\nu t}+\wtd z(t),t)-\dot{\wtd z}(t)|=|\wtd v^E(\xi\sqrt{\nu t}+\wtd z(t),t)- \wtd v^E(\wtd z(t),t)| \lesssim \sqrt{\nu t}|\xi|.
\]
The lemma follows, upon using $\nabla p = \frac12 \xi p(\xi)$.
\end{proof}

\begin{lemma} There holds
\[
\mathcal E_4(t)\lesssim t\lw(\| \wb (t)\|_{L^2_p}^2 + \| \nabla \wb(t)\|^2_{L^2_p}\rw).
\]
\end{lemma}
\begin{proof}
We write $\mathcal{E}_4(t)=-t\lw(\mathcal{E}_{41}(t)+\mathcal{E}_{42}(t)\rw)$, where 
\[
\begin{cases}
\mathcal{E}_{41}(t)&=\int_{\Rt}\lw(\vb\cdot\nabla w_{2,a}\rw)\wb(\xi,t)p(\xi)d\xi,\\
\mathcal{E}_{42}(t)&=\int_{\Rt}\lw(v_{2,a}\cdot \nabla \wb\rw)\wb(\xi,t)p(\xi)d\xi.\\
\end{cases}
\]
Using H\"older's inequality, we estimate 
\[\bega
|\mathcal{E}_{41}(t)|
&\le \|\vb(t)\|_{L^\infty} \| \wb(t)\|_{L^2_p} \lw(\int_{\Rt}|\nabla w_{2,a}(\xi,t)|^2p(\xi)d\xi\rw)^{1/2} ,
\enda
\]
in which the integral is bounded by Lemma \ref{lem-w2a}. As for $\|\vb(t)\|_{L^\infty}$, we use the elliptic estimate and Sobolev embedding, giving 
$$\|\vb\|_{L^\infty}^2 \lesssim \| \wb\|_{L^{4/3}}  \| \wb\|_{L^4} \lesssim \|\wb\|_{L^{4/3}} \| \wb \|_{L^2}^{1/2} ( \| \wb \|_{L^2} + \| \nabla \wb \|_{L^2})^{1/2}.$$
Recalling the weight function $p = e^{|\xi|^2/4}$, we have $\| \wb \|_{L^{4/3}} \lesssim \|\wb\|_{L^2_p}$. Thus, we get 
\beq\label{sup-vvv}
\|\vb\|_{L^\infty}^2 \lesssim \| \wb \|_{L^2_p}^{3/2} (\|\wb \|_{L^2_p} + \| \nabla \wb \|_{L^2_p})^{1/2} \lesssim \|\wb \|_{L^2_p}^2 + \| \nabla \wb \|_{L^2_p}^2,
\eeq
and so 
$$|\mathcal{E}_{41}(t)| \lesssim   \| \wb (t)\|_{L^2_p} (\| \wb (t)\|_{L^2_p} + \| \nabla \wb(t)\|_{L^2_p} )\lesssim \| \wb (t)\|_{L^2_p}^2 + \| \nabla \wb(t)\|^2_{L^2_p}.$$
On the other hand, the estimate on $\mathcal{E}_{42}(t)$ is direct, since $v_{2,a}$ is bounded. The lemma follows. 
\end{proof}

\begin{lemma}
There holds 
\[
\mathcal E_5(t)\lesssim t\lw(\|\wb(t)\|_{L^2_p}^2+\|\wb(t)\|_{L^2_p}^5 + \|\nabla \wb(t)\|_{L^2_p}^2\rw).\]
\end{lemma}
\begin{proof}
 By H\"older's inequality and \eqref{sup-vvv}, we get 
\[\bega
|\mathcal{E}_5(t)|
& = t\Big|\int_{\Rt}\lw(\vb\cdot \nabla\wb\rw)\wb(\xi,t)p(\xi)d\xi\Big|
\\&\le t\|\vb(t)\|_{L^\infty}\|\wb(t)\|_{L^2_p}\|\nabla \wb(t)\|_{L^2_p}\\
&\lesssim  t \lw(\|\wb(t)\|_{L^2_p}+\|\nabla \wb(t)\|_{L^2_p}\rw)^{1/4}\|\wb(t)\|_{L^2_p}^{7/4}\|\nabla \wb(t)\|_{L^2_p} .
\enda
\]
The lemma follows upon using Young's inequality. 
\end{proof}
\begin{lemma}
There holds 
\[
\mathcal{E}_6(t)\lesssim  \nu t \|\ve(t)\|_{L^\infty}^4 + \nu t\|\wb(t)\|^4_{L^2_p} + \nu \|\nabla \wb(t)\|^2_{L^2_p}.
\]
\end{lemma}
\begin{proof} Again by H\"older inequality, we get 
\[
\begin{aligned}
|\mathcal{E}_6(t)|
& = \nu \sqrt t\Big|\int_{\Rt}\lw( \ve\cdot \nabla\wb\rw)\wb(\xi,t)p(\xi)d\xi\Big|
\\
&\lesssim \nu t^{1/2}\|\ve(t)\|_{L^\infty}\|\wb(t)\|_{L^2_p}\|\nabla \wb(t)\|_{L^2_p},
\end{aligned}\]
which yields the lemma upon using Young's inequality. 
\end{proof}
\begin{lemma}
There holds 
\[
\mathcal{E}_7(t)\lesssim t^{1/2}\|\wb(t)\|_{L^2_p}.
\]
\end{lemma}
\begin{proof} Using the estimates from \eqref{error} for a fixed $\gamma\in\lw(\frac{1}{2},1\rw)$ and H\"older inequality, we get 
\[\bega
|\mathcal{E}_7(t)|&\le (\nu t)^{-1}\int_{\Rt}|\Phi_{\text{app}}(\xi,t)||\wb(\xi,t)|p(\xi)d\xi\\
&\le (\nu t)^{-1}\int_{\Rt}(\nu t^{3/2})C_\gamma e^{-\gamma |\xi|^2/4}|\wb(\xi,t)|p(\xi)d\xi\\
&\le C_\gamma t^{1/2} \lw(\int_{\Rt}e^{-2\gamma|\xi|^2/4}p(\xi)d\xi\rw)^{1/2}\lw(\int_{\Rt}|\wb(\xi,t)|^2p(\xi)d\xi\rw)^{1/2}\\
&\lesssim t^{1/2}\|\wb(t)\|_{L^2_p}, 
\enda
\]
where we used $\gamma>1/2$. This concludes the proof.
\end{proof}
\begin{lemma}There hold 
\[
\mathcal{E}_8(t)\lesssim t^{-1/2}\|\ve(t)\|_{L^\infty}\|\wb(t)\|_{L^2_p}, \qquad \mathcal{E}_9(t)\lesssim \nu t^{1/2}\|\ve(t)\|_{L^\infty}\|\wb(t)\|_{L^2_p}.
\]
\end{lemma}
\begin{proof}
We recall that 
\[
\mathcal{E}_8(t)=-\frac{1}{\sqrt{t}}\int_{\Rt}(\ve(\xi,t)\cdot \nabla G(\xi))\wb(\xi,t)p(\xi)d\xi
\]
where $G(\xi)=\frac{1}{4\pi} e^{-|\xi|^2/4}$ and $p(\xi) = e^{|\xi|^2/4}$. We have 
\[\bega
|\mathcal{E}_8(t)|&\lesssim  t^{-1/2}\|\ve(t)\|_{L^\infty}\int_{\Rt}|\xi| |\wb(\xi,t)|d\xi \lesssim  t^{-1/2}\|\ve(t)\|_{L^\infty} \| \wb(t)\|_{L^2_p}.
\enda
\]
The proof for $\mathcal{E}_9(t)$ is identical, upon recalling the pointwise bound on $\nabla w_{2,a}$ from Lemma \ref{lem-w2a}. 
\end{proof}
\begin{proof}[Proof of Proposition \ref{Eest}]\label{proof1}
We are now ready to prove Proposition \ref{Eest}. 
Collecting and combining all the estimates from the previous lemmas, we get
\beq\label{collect}
\bega
&t\frac{d}{dt}\|\wb(t)\|_{L^2_p}^2 +\kappa(\|(1+|\xi|)\wb(t)\|_{L^2_p}^2 + \|\nabla \wb(t)\|_{L^2_p}^2)\\
&\lesssim 
 t\lw(\|(1+|\xi|)\wb(t)\|_{L^2_p}^2+\|\wb(t)\|_{L^2_p}^5 + \|\nabla \wb(t)\|_{L^2_p}^2\rw) +  t^{1/2}\|\wb(t)\|_{L^2_p} 
 \\&\quad + \nu t \|\ve(t)\|_{L^\infty}^4 + \nu t\|\wb(t)\|^4_{L^2_p} + \nu \|\nabla \wb(t)\|^2_{L^2_p}+  t^{-1/2}\|\ve(t)\|_{L^\infty}\|\wb(t)\|_{L^2_p} ,
\enda
\eeq
for $\kappa = 1/24$. 
Taking $t$ and $\nu$ sufficiently small and using Young's inequality, we obtain 
\beq\label{collect1}
\bega
t\frac{d}{dt}\|\wb(t)\|_{L^2_p}^2  &+\frac \kappa2(\|(1+|\xi|)\wb(t)\|_{L^2_p}^2 + \|\nabla \wb(t)\|_{L^2_p}^2)\\
&\lesssim 
 t \|\wb(t)\|_{L^2_p}^5 + \nu t \|\ve(t)\|_{L^\infty}^4+  t^{-1}\|\ve(t)\|^2_{L^\infty}
 .
\enda
\eeq
This completes the proof of the proposition. 
\end{proof}

\begin{remark}\label{rem-time1} The constraint on the smallness of times $T$ is precisely due to the term $\mathcal{E}_3(t)$ treated in Lemma \ref{lem-E3}. The remaining terms are treated using the standard Young's inequality. Hence, we in fact obtain 
\beq\bega
t\frac{d}{dt}\|\wb(t)\|_{L^2_p}^2  &+\kappa \Big( \|\wb(t)\|_{L^2_p}^2 + \|\nabla \wb(t)\|_{L^2_p}^2 + (1-5t\|\nabla v^E(t)\|_{L^\infty})\|\xi \wb(t)\|_{L^2_p}^2\Big)\\
&\lesssim 
 t ( \|\wb (t)\|_{L^2_p}^2 + \|\wb(t)\|_{L^2_p}^5) + \nu t \|\ve(t)\|_{L^\infty}^4+  t^{-1}\|\ve(t)\|^2_{L^\infty}
\enda
\eeq
for all positive times, as long as the estimates from Proposition \ref{prop-tWE} and Corollary  \ref{cor-tWE} on the approximate vortex-wave solutions are valid. This yields a lower bound on the smallness of $T$ so that $\sup_{0\le t\le T}5 t\|\nabla v^E (t)\|_{L^\infty} \le1$. 
\end{remark}

\begin{remark}
One may try to improve the time interval by introducing a new weight function, as done similarly in \cite{Gallay11}, $p_{new}(\xi) = p(\xi) (1+  \nu t q(\xi,t))$, where $q(\xi,t)$ solves
$$ v^G(\xi) \cdot \nabla_\xi q =  \frac{1}{\sqrt{\nu t}} \Big( v^E(z(t) + \xi \sqrt{\nu t},t) - v^E(z(t),t) \Big) \cdot \xi ,$$
whose solution is however unclear for large $\xi \sqrt{\nu t}$.
\end{remark}

\section{Inviscid limit for the regular part}\label{sec-regular}
In the previous section, we have proved the apriori estimate for $\w^{B,\nu}$ and $v^{E,\nu}$ in the weighted energy space with the re-scaled variable $\xi=\frac{x-\wtd z(t)}{\sqrt{\nu t}}$. In this section, we derive estimates on the regular vorticity component $\w^{E,\nu}$, which solves 
\beq \label{temp}
\pt_t\w^{E,\nu}+u^\nu\cdot \nabla \w^{E,\nu}=\nu\triangle \w^{E,\nu}
\eeq
 with the initial data $\w_0^E$. We write:
\beq \label{forms}
\begin{cases}
\w^{E,\nu}(t,x)&=\wtd \w^E(t,x)+\nu^{3/2}\we(t,x) , \\
v^{E,\nu}(t,x)&= \wtd v^E(t,x)+\nu^{3/2}\ve(t,x), \\
v^{B,\nu}(t,x)&=\frac{1}{\sqrt{\nu t}}v^G\lw(\frac{x-\wtd z(t)}{\sqrt{\nu t}}\rw) + \sqrt{\nu t} (v_{2,a}+\vb)\lw(\frac{x-\wtd z(t)}{\sqrt{\nu t}},t\rw),\\
u^{\nu}(t,x)&=v^{E,\nu}(t,x)+v^{B,\nu}(t,x),\\
\end{cases}
\eeq
where $(\wtd z(t),\wtd \w^E)$ is the solution to the viscous vortex-wave system introduced in Section \ref{sec-appVW}, while $v^G$ and $v_{2,a}$ are constructed in Section \ref{sec3}. Here, we note that the form of the common velocity $u^{\nu}(t,x)$ is compatible with the form in \eqref{remainder} and \eqref{change} in the scaled variable $\xi$. The velocity $\vb$ is kept the same as in the previous section, with $\xi$ is replaced by $\frac{x-\wtd z(t)}{\sqrt{\nu t}}$ and $\vb=K\star_\xi \wb$. It is natural to work in the original variables $(x,t)$ instead of $(\xi,t)$, since $\w^{E,\nu}(t)$ solves \eqref{temp} with regular initial data $\w_0^E$. Hence one does not expect $\w^{E,\nu}$ to have the localized behavior near the point vortex. 
Roughly speaking, we want to get an apriori bound on $\|\ve(t)\|_{L^\infty}$ (in terms of $\wb(t)$) on a time interval independent of $\nu$.
Precisely, we shall prove the following proposition.  

\begin{proposition} \label{goal-est} Let $\we$ solve the equations  \eqref{temp} and \eqref{forms}. 
There exists a positive time $T$, independent of $\nu>0$, such that 
\[
\|\we(t)\|_{\lbb}\lesssim \int_0^t s^{3/2}(\|\wb(t)\|_{L^2_p}+\|\nabla \wb(t)\|_{L^2_p})ds+\nu^{1/2} t\]
for $ t \in [0,T ]$. 
\end{proposition}

\subsection{Equations for the remainder}\label{heat-trans}
In this subsection, we derive the equations for the remainder $\we$ as well as $\vb$ appearing in \eqref{temp} and \eqref{forms}.
Putting the Ansatz \eqref{forms} into equation \eqref{temp} and using equation \eqref{w1a}, we obtain the following transport-diffusion equation for $\we$:
\[
\pt_t\we+u^\nu\cdot \nabla \we-\nu \triangle \we=f(x,t),
\]
where $f(x,t)$ are given by 
\beq\label{wee}
\begin{aligned}
f(x,t)&= - \frac{1}{\nu \sqrt{t}}\lw(v^G\lw(\frac{x-\wtd z(t)}{\sqrt{\nu t}}\rw)-v^G\lw(\frac{x- z(t)}{\sqrt{\nu t}}\rw)\rw)\cdot \nabla w_{1,a} - \ve\cdot \nabla \wtd\w^E - \frac{\sqrt{t}}{\nu}\vb\cdot \nabla \wtd \w^E\\
&\quad - \sqrt{\nu}(v_{1,a}\cdot\nabla w_{1,a}) + \frac{1}{2\pi \nu^{3/2}}\frac{(x- z(t))^\perp}{|x- z(t)|^2}e^{-\frac{|x- z(t)|^2}{4\nu t}}\cdot \nabla  \w^E+\sqrt{\nu}\triangle w_{1,a}.\\
\end{aligned}
\eeq 
\subsection{Estimating the forcing term $f(x,t)$}
In this subsection, we prove the following proposition
\begin{proposition}\label{forcing}
Let $f(x,t)$ be defined as in \eqref{wee}. There holds 
\[
\|f(t)\|_{\lbb}\lesssim \|\we(t)\|_{\lbb}+t^{3/2}\lw(\|\wb(t)\|_{L^2_p}+\|\nabla \wb(t)\|_{L^2_p}\rw) + \sqrt{\nu}.
\]
\end{proposition}

We will give a proof at the end of this subsection, after proving some useful lemmas. First, let us write $f$ as:
\[
f(x,t)=f_1(x,t)+f_2(x,t)+f_3(x,t)
\]
where
\[
\begin{cases}
f_1(x,t)&=-\frac{1}{\nu \sqrt{t}}\lw(v^G\lw(\frac{x-\wtd z(t)}{\sqrt{\nu t}}\rw)-v^G\lw(\frac{x- z(t)}{\sqrt{\nu t}}\rw)\rw)\cdot \nabla w_{1,a}  - \sqrt{\nu}(v_{1,a}\cdot\nabla w_{1,a})  \\&\quad +\frac{1}{2\pi \nu^{3/2}}\frac{(x- z(t))^\perp}{|x- z(t)|^2}e^{-\frac{|x- z(t)|^2}{4\nu t}}\cdot \nabla  \w^E+\sqrt{\nu}\triangle w_{1,a},\\
f_2(x,t)&= - \ve\cdot \nabla \wtd \w^E,\\
f_3(x,t)&= - \frac{\sqrt{t}}{\nu}\vb\cdot \nabla \wtd \w^E,\\
\end{cases}
\]
In what follows, we bound $\|f_i(t)\|_{\lbb}$ for each $i\in \{1,2,3\}$.

\begin{lemma} There holds 
\[
\|f_1(t)\|_{\lbb}\lesssim \sqrt{\nu}
\]
uniformly in $\nu>0$.
\end{lemma}
\begin{proof}
First we see that  
\[
\lw\|- \sqrt{\nu}(v_{1,a}\cdot\nabla w_{1,a}) - \frac{1}{2\pi \nu^{3/2}}\frac{(x- z(t))^\perp}{|x- z(t)|^2}e^{-\frac{|x- z(t)|^2}{4\nu t}}\cdot \nabla  \w^E+\sqrt{\nu}\triangle \w_{1,a}\rw\|_{\lbb}\lesssim \sqrt{\nu}
\]
thanks to the fact that $\w^E$ is supported away from $z(t)$ and $\wtd z(t)$, and $ w_{1,a}$ is bounded in $W^{2,4}$, by Proposition \ref{prop-tWE}.
Now for the first term in $f_1$, it suffices to prove that 
\beq\label{vG}
\frac{1}{ \sqrt{\nu t}}\lw|v^G\lw(\frac{x-\wtd z(t)}{\sqrt{\nu t}}\rw)-v^G\lw(\frac{x- z(t)}{\sqrt{\nu t}}\rw)\rw|\lesssim\nu  t\qquad \text{for all}\quad x\in \text{supp}(w_{1,a}).
\eeq
As long as the above claim is proved, we would  get \[
\bega
&\lw\| \frac{1}{\nu \sqrt{t}}\lw(v^G\lw(\frac{x-\wtd z(t)}{\sqrt{\nu t}}\rw)-v^G\lw(\frac{x- z(t)}{\sqrt{\nu t}}\rw)\rw)\cdot \nabla w_{1,a}\rw\|_{\lbb}
\\&\qquad\lesssim \sqrt \nu \|\nabla w_{1,a}(t)\|_{\lbb(\text{supp}(w_{1,a}))}\lesssim \sqrt \nu
\enda
\]
by Proposition \ref{prop-tWE}.

Now we shall prove the inequality \eqref{vG}. To this end, let us denote \beq\label{nota1}\eta_1=x-\wtd z(t),\qquad \text{and}\quad \eta_2=x-z(t)\eeq
The left hand side of \eqref{vG} can be re-written as:
\beq \label{est900}\frac{1}{\sqrt{\nu t}}\lw(v^G(\frac{\eta_1}{\sqrt{\nu t}})-v^G(\frac{\eta_2}{\sqrt{\nu t}})\rw)=\frac{1}{2\pi}\lw(V_1(\eta_1,\eta_2)+V_2(\eta_1,\eta_2)\rw)\eeq
where\[\begin{cases}V_1(\eta_1,\eta_2)&=\lw(\frac{\eta_1^\perp}{|\eta_1|^2}-\frac{\eta_2^\perp}{|\eta_2|^2}\rw),\\V_2(\eta_1,\eta_2)&=\lw(\frac{\eta_2^\perp}{|\eta_2|^2}e^{-|\eta_2|^2/4\nu t}-\frac{\eta_1^\perp}{|\eta_1|^2}e^{-|\eta_1|^2/4\nu t}\rw).\\\end{cases}\]

When $x\in \text{supp}(\wtd \w^E(t))$, by the properties established in Section \ref{sec-appVW}, we have a positive constant $c_T$, independent of $\nu$, such that  \beq\label{est84}|x-z(t)|\ge c_T\qquad \text{and}\qquad |x-\wtd z(t)|\ge c_T\qquad \forall t\in [0,T].\eeq
This implies that $|\eta_1|\ge c_T$ and $|\eta_2|\ge c_T$, upon recalling the notations \eqref{nota1}.
Thus, we get 
\[\bega|V_1(\eta_1,\eta_2)|&=\lw|\frac{\eta_1^\perp}{|\eta_1|^2}-\frac{\eta_2^\perp}{|\eta_2|^2}\rw|\le \lw|\frac{\eta_1^\perp}{|\eta_1|^2}-\frac{\eta_2^\perp}{|\eta_1|^2}\rw|+\lw|\frac{\eta_2^\perp}{|\eta_1|^2}-\frac{\eta_2^\perp}{|\eta_2|^2}\rw|
\\&\le \frac{|\eta_1-\eta_2|}{|\eta_1|^2}+|\eta_2|\frac{\lw||\eta_2|^2-|\eta_1|^2\rw|}{|\eta_1|^2|\eta_2|^2}
\\&\le c_T^{-2}|\eta_1-\eta_2|+\frac{1}{|\eta_1|^2|\eta_2|}\lw||\eta_2|-|\eta_1|\rw|(|\eta_1|+|\eta_2|)\lesssim |\eta_1-\eta_2|\\
&=|(x-\wtd z(t))-(x-z(t))|=|\wtd z(t)-z(t)|\lesssim \nu t\qquad (\text{by the estimate \eqref{est-z2}}).
\enda
\]
Hence
\beq\label{est81}|V_1(\eta_1,\eta_2)|\lesssim \nu t.\eeq
Now for $V_2(\eta_1,\eta_2)$, note that we shall only consider $x\in\text{supp}(\wtd \w^E(t))$, in which we get \eqref{est84}. In this case we get \beq \label{est901}\bega|V_2(\eta_1,\eta_2)|&\le |\eta_2|^{-1}e^{-|\eta_2|^2/4\nu t}+|\eta_1|^{-1}e^{-|\eta_1|^2/4\nu t}\le 2 c_T^{-1} e^{-c_T^2/4\nu t} \lesssim \nu t\enda\eeq
Combining \eqref{est900}, \eqref{est81} and \eqref{est901}, we get the desired inequality \eqref{vG}. The bound for the first term is complete. This concludes the proof.
\end{proof}
\begin{lemma}There holds 
\[
\|f_2(t)\|_{\lbb}\lesssim \|\we(t)\|_{\lbb}
\]
\end{lemma}
\begin{proof}
We have
\[
\|f_2(t)\|_{\lbb}=\|\ve(t)\cdot \nabla\wtd \w^E(t)\|_{\lbb}\le \|\ve(t)\|_{L^\infty}\|\nabla \wtd \w^E(t)\|_{\lbb}\lesssim \|\we(t)\|_{\lbb}
\]
by Corollary \ref{cor-tWE} and Lemma \ref{lem-velocity}. The proof is complete.
\end{proof}
\begin{lemma}
There holds 
\[
\|f_3(t)\|_{\lbb}\lesssim t^{3/2}\lw(\|\wb(t)\|_{L^2_p}+\|\nabla \wb(t)\|_{L^2_p}\rw)
\]
\end{lemma}
\begin{proof}
We recall that 
 \[
 f_3(x,t)=\frac{\sqrt{t}}{\nu}\vb\lw(\xi,t\rw)\cdot \nabla\wtd \w^E(t,x),\qquad \xi=\frac{x-\wtd z(t)}{\sqrt{\nu t}}.
  \]
We shall only consider $x\in \text{supp}(\wtd \w^E(t))$. Since $\wtd \w^E(t)$ is supported away from $\wtd z(t)$, there exists $d_T>0$ such that 
\beq\label{away12}
|x-\wtd z(t)|\ge d_T\qquad \text{for}\quad x\in\text{supp}(\wtd \w^E(t)).
\eeq
Since $\int_{\Rt}\wb(\xi,t)d\xi=0$, by Lemma \eqref{lem-velocity}, we get 
\[\bega
\|(1+|\xi|^2)\vb(t)\|_{L^\infty}&\lesssim \|(1+|\xi|^2)\wb(t)\|_{L^4}+\|(1+|\xi|^2)\wb(t)\|_{L^{4/3}}\\
&\lesssim \|\wb(t)\|_{L^2_p}+\|\nabla \wb(t)\|_{L^2_p}.\\
\enda
\]
This implies that, for $x$ in the support of $\wtd \w^E(t)$, we get 
\[
\bega
|\vb(t,\xi)|&\lesssim \frac{1}{1+|\xi|^2}\lw(\|\wb(t)\|_{L^2_p}+\|\nabla \wb(t)\|_{L^2_p}\rw)
\lesssim (\nu t) \lw(\|\wb(t)\|_{L^2_p}+\|\nabla \wb(t)\|_{L^2_p}\rw).\\
\enda
\]
Thus we get 
\[
\|f_3(t)\|_{\lbb}\lesssim \frac{\sqrt{t}}{\nu} \|\vb\lw(\xi,t\rw)\cdot \nabla\wtd \w^E(t,x)\|_{\lbb} \lesssim t^{3/2}\lw(\|\wb(t)\|_{L^2_p}+\|\nabla \wb(t)\|_{L^2_p}\rw).
\]
The proof is complete.
\end{proof}
 We conclude this subsection by proving the Proposition \ref{forcing}
\begin{proof}[Proof of Proposition \ref{forcing}]
The proof follows as a direct consequence of the previous lemmas for $f_i,\quad i\in \{1,2,3\}$ in this subsection.
\end{proof}

\subsection{Apriori estimates for the remainder}
In this section, we give a proof for our main Theorem \ref{goal-est}, stated at the beginning of this subsection. We recall from Section \ref{heat-trans} that $\we$ solves the heat transport equation
\[
\pt_t \we+ u^\nu\cdot \nabla \we-\nu \triangle \we=f(x,t) .
\]
A standard $L^4\cap L^{4/3}$ estimate for the heat transport equation yields 
\[\bega
\dfrac{d}{dt}\lw(\|\we(t)\|_{\lbb}\rw)&\lesssim \|f(t)\|_{\lbb}\\
&\lesssim \|\we(t)\|_{\lbb}+t^{3/2}\lw(\|\wb(t)\|_{L^2_p}+\|\nabla \wb(t)\|_{L^2_p}\rw)+\sqrt{\nu},\\
\enda
\]
using Proposition \ref{forcing}. Now applying Gronwall lemma for the above inequality, we have 
\beq\label{est8}
\bega
\|\we(t)\|_{\lbb}&\lesssim \int_0^t \lw(s^{3/2}(\|\wb(t)\|_{L^2_p}+\|\nabla \wb(t)\|_{L^2_p})+\sqrt\nu \rw)ds\\
&\lesssim \int_0^t s^{3/2}(\|\wb(t)\|_{L^2_p}+\|\nabla \wb(t)\|_{L^2_p})ds+\nu^{1/2} t.\\
\enda
\eeq
The proof is complete.

\section{Proof of inviscid limit}

In this section, we conclude the proof for inviscid limit, using the apriori estimates obtained from the previous sections. 
Let us first prove the following proposition, before proving our main theorem, stated in the first part of this paper. 
\begin{proposition}
There exists a time $T>0$, independent of the viscosity $\nu$, such that 
\[
\sup_{0\le t\le T}\lw(\|\wb(t)\|_{L^2_p}+\|\we(t)\|_{\lbb}\rw)\lesssim 1,
\]
uniformly in $\nu$. \end{proposition}
\begin{proof}
First, we recall the following estimates for $\|\wb(t)\|_{L^2_p}$
and $\|\we(t)\|_{\lbb}$ proven in Propositions \ref{Eest} and \ref{goal-est}.
\beq \label{est9}
\begin{aligned}
&\frac{d}{dt}\|\wb(t)\|_{L^2_p}^2 +\frac{\kappa}{t}(\|(1+|\xi|)\wb(t)\|_{L^2_p}^2 + \|\nabla \wb(t)\|_{L^2_p}^2)\lesssim \|\wb(t)\|_{L^2_p}^5 + \nu  \|\ve(t)\|_{L^\infty}^4+  t^{-2}\|\ve(t)\|^2_{L^\infty}\\
& \|\we(t)\|_{\lbb}\lesssim \int_0^t s^{3/2}\lw(\|\wb(t)\|_{L^2_p}+\|\nabla \wb(t)\|_{L^2_p}\rw)ds+\nu^{1/2}t.\\
\end{aligned}
\eeq
Let 
$$\mathcal{G}(t)=\|\wb(t)\|_{L^2_p}^2+\int_0^t s^{-1}(\|\wb(s)\|^2_{L^2_p}+\|\nabla \wb(s)\|^2_{L^2_p})ds.$$
From the inequality \eqref{est9}, it is straight-forward that 
\beq \label{str}
\|\we(t)\|_{\lbb}\lesssim t^{5/2} \mathcal{G}(t)^{1/2}+\nu^{1/2}t .
\eeq
Thus, we have 
\[\bega
\mathcal{G}'(t)&=\frac{d}{dt}\|\wb(t)\|_{L^2_p}^2+t^{-1}\lw(\|\wb(t)\|^2_{L^2_p}+\|\nabla \wb(t)\|^2_{L^2_p}\rw)\\
&\lesssim \|\wb(t)\|_{L^2_p}^5+\nu \|\ve(t)\|_{L^\infty}^4+t^{-2}\|\ve(t)\|_{L^\infty}^2\qquad (\text{by \eqref{est9}})\\
&\lesssim \mathcal{G}(t)^{5/2}+\nu \|\we(t)\|_{\lbb}^4+t^{-2}\|\we(t)\|_{\lbb}^2\\
&\lesssim \mathcal{G}(t)^{5/2}+\nu \lw( t^{5/2} \mathcal{G}(t)^{1/2}+\nu^{1/2}t\rw)^4+t^{-2}\lw(t^{5/2} \mathcal{G}(t)^{1/2}+\nu^{1/2}t\rw)^2\qquad (\text{by \eqref{str}})\\
&\lesssim \mathcal{G}(t)^{5/2}+\nu t^{10} \mathcal G(t)^2+\nu^3 t^4+t^3 \mathcal{G}(t)+\nu.\\
\enda
\]
By standard ODE theory, we have a time $T>0$, which is independent of $\nu>0$, such that $\mathcal G(t)$ is uniformly bounded for $t\in [0,T]$. Since $\mathcal G(t)\ge \|\wb(t)\|_{L^2_p}$, the proof for $\|\wb(t)\|_{L^2_p}$ is complete. The bound $\|\we(t)\|_{\lbb}\lesssim 1$ follows from the inequality \eqref{str}.
\end{proof}
We conclude this section by proving our main theorem, stated in the first part of this paper.
\begin{proof}[Proof of theorem \ref{mainthm}]
We have proved that $\|\wb(t)\|_{L^2_p}$ is uniformly bounded in $\nu$. We recall from Section \ref{sec3} that 
\[
\w^{B,\nu}(t,x)=\frac{1}{\nu t}w_2(\xi,t)=\frac{1}{\nu t}\lw(G(\xi)+(\nu t)w_{2,a}+(\nu t)\wb\rw)=\frac{1}{\nu t}G(\xi)+w_{2,a}+\wb ,
\] 
where $G(\xi) = \frac{1}{4\pi} e^{-|\xi|^2/4}$ and $\xi = (x-\wtd z(t))/\sqrt{\nu t}$.  
We compute 
\beq\label{st1}
\bega
\lw\|\w^{B,\nu}(t,x)-\frac{1}{4\pi\nu t}e^{-\frac{|x-\wtd z(t)|^2}{4\nu t}}\rw\|_{L^1_x}& =\lw\|w_{2,a}(\xi,t)+\wb(\xi,t)\rw\|_{L^1_x}\\
&=\nu t\|w_{2,a}(t)+\wb(t)\|_{L^1_\xi}\lesssim (\nu t)\lw(\|w_{2,a}(t)\|_{L^2_p}+\|\wb(t)\|_{L^2_p}\rw)\\
&\lesssim (\nu t).
\enda
\eeq
For simplicity of notations, we denote by $G_{\wtd z(t)}(x)$ and $G_{z(t)}(x)$ the Gaussians $\frac{1}{4\pi\nu t}e^{-\frac{|x-\wtd z(t)|^2}{4\nu t}}$ and $\frac{1}{4\pi\nu t}e^{-\frac{|x-z(t)|^2}{4\nu t}}$, respectively. Our goal now is to compare the two Gaussians in $L^1$ norm. To this end, let us denote $A =\frac{|x-\wtd z(t)|^2}{4\nu t}$ and $B = \frac{|x- z(t)|^2}{4\nu t}$. We have 
\[
G_{\wtd z(t)}(x)-G_{z(t)}(x)=e^{-A}-e^{-B}=e^{-B}\lw(e^{B-A}-1\rw).
\]
We have 
\[\bega
B-A&=(4\nu t)^{-1}\lw(|x-z(t)|^2-|x-\wtd z(t)|^2\rw)=(4\nu t)^{-1}\lw(2x\cdot(\wtd z(t)-z(t))+|z(t)|^2-|\wtd z(t)|^2\rw)\\
&\lesssim (4\nu t)^{-1}\lw(|x||\wtd z(t)-z(t)|+|\wtd z(t)-z(t)|\rw)\\
&\lesssim |x|+1\qquad (\text{since}\quad |\wtd z(t)-z(t)|\lesssim \nu t)\\
&\le |x-z(t)|+|z(t)|+1\lesssim \frac{|x-z(t)|}{\sqrt{\nu t}}+1.
\enda
\]
Here we used the standard fact of the vortex-wave system that $|z(t)|\lesssim 1$ for any fixed interval of time. For, one can see that $|z(t)|\le |z_0|+\int_0^t |v^E(z(s),s)|ds\le |z_0|+t\|v^E\|_{L^\infty}.$ Hence we get
\beq \label{temp10000}
\bega
|G_{\wtd z(t)}(x)-G_{z(t)}(x)|&\lesssim e^{-\frac{|x-z(t)|^2}{4\nu t}+M_T\frac{|x-z(t)|}{\sqrt{\nu t}}}\qquad \text{for some}\quad M_T>0.\\
\enda
\eeq
Integrating both sides of the inequality \eqref{temp10000} in $x\in\Rt$, we have 
\[
\|G_{z(t)}-G_{\wtd z(t)}\|_{L^1_x}\lesssim \int_{\Rt}e^{-\frac{|x-z(t)|^2}{4\nu t}+M_T\frac{|x-z(t)|}{\sqrt{\nu t}}}dx.
\]
Making the change of variables $y=\frac{x-z(t)}{\sqrt{\nu t}}$ in the above integral, we thus obtain 
\beq\label{st2}
\|G_{z(t)}-G_{\wtd z(t)}\|_{L^1_x}\lesssim \nu t.
\eeq
Combining the inequalities \eqref{st1} and \eqref{st2}, we get 
\[
\lw\|\w^{B,\nu}(t,x)-\frac{1}{4\pi\nu t}e^{-\frac{|x-z(t)|^2}{4\nu t}}\rw\|_{L^1_x}\lesssim \nu t.
\]
The inequality $\|\w^{E,\nu}(t)-\w^E(t)\|_{\lbb}\lesssim \nu$ follows directly from the expansion \eqref{forms},  the inequality \eqref{str} and the uniform bound of $\mathcal G(t)$. The proof is complete.  
\end{proof}

\appendix 
\section{Appendix}

In this section, we collect several useful lemmas used in this paper.
\begin{lemma}[Elliptic estimates]\label{lem-velocity}
Let $v=K\star \w$ be the velocity vector field obtained from the vorticity $\w$ on $\Rt$. Define the norm $\|\cdot\|_{\lbb}=\|\cdot\|_{L^4}+\|\cdot\|_{L^{4/3}}$. There hold the following inequalities
\[
\bega
\|v\|_{L^\infty}&\lesssim \|\w\|_{\lbb},\qquad \|v\|_{L^\infty}\lesssim \|\w\|_{L^1\cap L^\infty} .
\enda
\]
Moreover, if $\int_{\Rt}\w(x)dx=0$, then 
\[
\| (1+|x|^2)v\|_{L^\infty} \lesssim \|(1+|x|^2)\w\|_{\lbb}.
\]
\end{lemma}
\begin{proof}
From the Biot-Savart law \eqref{BS-law}, we estimate 
\beq \label{repeat}
\bega
|v(x)|&\lesssim \int_{\Rt}\frac{|\w(y)|}{|x-y|}dy=\lw(\int_{|x-y|\le R}+\int_{|x-y|\ge R}\rw)\frac{|\w(y)|}{|x-y|}dy\\
&\lesssim \lw(\int_{|x-y|\le R}|x-y|^{-4/3}dy\rw)^{3/4}\|\w\|_{L^4}+\lw(\int_{|x-y|\ge R}|x-y|^{-4}dy\rw)^{1/4}\|\w\|_{L^{4/3}}\\
&\lesssim R^{1/2}\|\w\|_{L^4}+R^{-1/2}\|\w\|_{L^{4/3}}.\\
\enda
\eeq
Thus choosing $R=\frac{\|\w\|_{L^{4/3}}}{\|\w\|_{L^4}}$, we have $\|v\|_{L^\infty}\lesssim \|\w\|_{L^{4/3}}^{1/2}\|\w\|_{L^4}^{1/2} $, which gives the first inequality. As for the second, we use $\| \omega \|_{L^p} \le \| \omega \|_{L^1}^{1/p}\|\omega\|^{1-1/p}_{L^\infty}$.

It remains to check the last inequality. We shall check it for $v_2$, the second component of $v$. The estimate on $v_1$ is similar. First, we check \beq \label{est20}
|x||v_2(x)|\lesssim \int_{\Rt}\frac{1}{|x-y|}|y||\w(y)|dy.
\eeq
By Biot-Savart law and $\int_{\Rt}\w(y)dy=0$, we have 
\[
|v_2(x)|=\frac{1}{2\pi}\lw|\int_{\Rt}\frac{x_1-y_1}{|x-y|^2}\w(y)dy\rw|\le \frac{1}{2\pi}\int_{\Rt}\lw|\frac{x_1-y_1}{|x-y|^2}-\frac{x_1}{|x|^2}\rw||\w(y)|dy.\\
\]
Now we have 
\[
\frac{x_1-y_1}{|x-y|^2}-\frac{x_1}{|x|^2}=\frac{1}{|x|^2|x-y|^2}\lw(|x|^2(x_1-y_1)-x_1|x-y|^2\rw).
\]
It follows that $|x|^2(x_1-y_1)-x_1|x-y|^2\le 4|x||y||x-y|$. Hence, 
\[
|x|\Big[ \frac{x_1-y_1}{|x-y|^2}-\frac{x_1}{|x|^2} \Big]\le  \frac{4|y|}{|x-y|},
\]
which gives \eqref{est20}. Now multiplying both sides of \eqref{est20} by $|x|$, we have 
\[\bega
|x|^2|v_2(x)|&\lesssim \int_{\Rt}\frac{|x||y|}{|x-y|}|\w(y)|dy\le \int_{\Rt}\frac{|y|+|x-y|}{|x-y|}|y||\w(y)|dy.\\
&=\int_{\Rt}\frac{1}{|x-y|}|y|^2|\w(y)|dy+\int_{\Rt}|y||\w(y)|dy\\
\enda
\]
Let us first treat the first term in the above. Repeating the argument of \eqref{repeat} for $\w=|y|^2|\w(y)|$, we have
\[
\int_{\Rt}\frac{1}{|x-y|}|y|^2|\w(y)|dy\lesssim \|(1+|y|^2)\w(y)\|_{\lbb}.
\]
For the second term, using H\"older inequality, we get
\[
\int_{\Rt}|y||\w(y)|dy=\int_{\Rt}\frac{|y|}{1+|y|^2}(1+|y|^2)||\w(y)|dy\lesssim\|(1+|y|^2)|\w(y)\|_{L^{4/3}}.
\]
Thus 
\[
|x|^2|v_2(x)|\lesssim \|(1+|x|^2)\w\|_{\lbb}.
\]
The lemma follows.
\end{proof}

\begin{lemma}\label{exp}
Let $z_1,z_2\in \mathbb{C}$ and $\psi$ be the angle between $z_1$ and $z_2$. Assuming that $|z_1|<|z_2|$ and $\sin(\psi)\neq 0$, there holds 
\[
\dfrac{1}{|z_1+z_2|^2}-\dfrac{1}{|z_2|^2}=\dfrac{1}{|z_2|^2}\sum_{n=1}^\infty(-1)^n \dfrac{|z_1|^n}{|z_2|^n}\dfrac{\sin((n+1)\psi)}{\sin(\psi)}.
\]
\end{lemma}
\begin{proof}
Let $\frac{z_1}{z_2}=z=re^{i\psi}$. We have 
\[
\frac{1}{|z_1+z_2|^2}-\frac{1}{|z_2|^2}=\frac{1}{|z_2|^2}\lw(\frac{1}{|1+z|^2}-1\rw).
\]
Now for $|z|<1$, we have 
\[\bega
\frac{1}{|1+z|^2}&=\frac{1}{(1+z)(1+\bar z)}=(1-z+z^2-\cdots)(1-\bar z+\bar z^2+\cdots)\\
&=1-(z+\bar z)+(z^2+z\bar z+\bar z^2)-(z^3+z^2\bar z+z\bar z^2+\bar z^3)+\cdots .
\enda
\]
Now for each $n$, we have 
\[
z^n+z^{n-1}\bar z+\cdots+z\bar z^{n-1}+\bar z^{n}=\dfrac{z^{n+1}-\bar z^{n+1}}{z-\bar z}=r^n\dfrac{\sin((n+1)\psi)}{\sin \psi}.
\]
This concludes the proof.
\end{proof}

%

%

\end{document}